\documentclass{amsart} \usepackage{latexsym,amsxtra,amscd,ifthen} \usepackage{amsfonts} \usepackage{verbatim}
\usepackage{amsmath} \usepackage{amsthm} \usepackage{amssymb}

\input xy \xyoption{matrix} \xyoption{arrow}\xyoption{frame} 
 
 \newcommand{\edge}{\ar@{-}}

\numberwithin{equation}{section}

\theoremstyle{plain} \newtheorem{theorem}{Theorem}[section] \newtheorem{lemma}[theorem]{Lemma}
\newtheorem{proposition}[theorem]{Proposition} \newtheorem{corollary}[theorem]{Corollary}

\theoremstyle{definition} \newtheorem{definition}[theorem]{Definition} \newtheorem{example}[theorem]{Example}

\newtheorem{remark}[theorem]{Remark} \newtheorem{remarks}[theorem]{Remarks} \newtheorem*{remark*}{Remark} \newtheorem{question}[theorem]{Question}

\newcommand{\gr}{\operatorname{gr}}
\newcommand{\GK}{\operatorname{GKdim}}

\newcommand{\gnoc}{\mathrel{{\lower.2ex\hbox{$\backsim$}}\llap{\raise.45ex\hbox{=}}}}

\begin{document}

\title[] {Quantum homogeneous spaces of connected Hopf algebras}

\author[Ken Brown and Paul Gilmartin]{Ken Brown and Paul Gilmartin} \address{School of Mathematics and Statistics\\ University of Glasgow\\ Glasgow G12 8QW\\
Scotland.} \email{Ken.Brown@glasgow.ac.uk; p.gilmartin.1@research.gla.ac.uk}


\thanks{Some of these results will form part of the second author's PhD thesis at the University of Glasgow, supported by a Scholarship of the Carnegie Trust. Both authors thanks James Zhang for very helpful discussions.}


\maketitle

\begin{abstract} Let $H$ be a connected Hopf $k$-algebra of finite Gel'fand-Kirillov dimension over an algebraically closed field $k$ of characteristic 0. The objects of study in this paper are the left or right coideal subalgebras $T$ of $H$. They are shown to be deformations of commutative polynomial $k$-algebras. A number of well-known homological and other properties follow immediately from this fact. Further properties are described, examples are considered,  invariants are constructed and a number of open questions are listed.

\end{abstract}

\section{Introduction}

\subsection{} \label{intro1} The left and right coideal subalgebras of a Hopf algebra $H$ (defined in $\S$\ref{coideal}) have been an important focus of research since the classical work on the commutative and cocommutative cases in the last century, \cite{DG}, \cite{Tak}, \cite{New}. Following the seminal papers of Takeuchi \cite{Takeuchi}, Masuoka \cite{Masuoka} and Schneider \cite{Schn} attention has concentrated on the \emph{quantum homogeneous spaces}, that is those coideal subalgebras of $H$ over which $H$ is a faithfully flat module. This paper continues this line of research for the case where $H$ is a connected Hopf $k$-algebra of finite Gel'fand-Kirillov dimension $n$, (written $\mathrm{GKdim}H = n$), with $k$ an algebraically closed field of characteristic 0.

This class of Hopf algebras has been the subject of a series of recent papers, see for example \cite{Zhuang}, \cite{OHagan}, \cite{WZZ4}. None of these, however, has examined their coideal subalgebras, so a primary aim here is to lay out their basic properties and clarify topics for future research. Motivating examples from the classical theory are plentiful and offer a rich source of intuition - thus, $H$ as above is \emph{commutative} if and only if it is the coordinate ring $\mathcal{O}(U)$ of a unipotent algebraic $k$-group $U$ of dimension $n$, while $H$ is \emph{cocommutative} if and only if it is the enveloping algebra $U(\mathfrak{g})$ of its $n$-dimensional Lie $k$-algebra $\mathfrak{g}$ of primitive elements. In the former case the coideal subalgebras of $H$ are the (right and left) homogeneous $U$-spaces; in the latter case - thanks to cocommutativity - the coideal subalgebras are the Hopf subalgebras, hence are just the enveloping algebras of the Lie subalgebras of $\mathfrak{g}$. For references for these classical facts, see $\S\ref{commutative},\,\S\ref{cocommutative}$.

\subsection{} \label{intro2} In both the above classical settings the faithful flatness condition always holds, and both the Hopf algebra $H$ and its coideal subalgebras either are themselves (in the first case), or have associated graded algebras which are (in the second case) commutative polynomial $k$-algebras. To extend this picture to the non-classical world, we need the concept of the \emph{coradical filtration} of a Hopf algebra, recalled in $\S$\ref{coradfilt}. The starting point is then the result of Zhuang, \cite[Theorem 6.9]{Zhuang}, restated below as Theorem \ref{beautiful}, stating that the associated graded algebra $\mathrm{gr}H$ of $H$ with respect to its coradical filtration is a polynomial $k$-algebra in $n = \mathrm{GKdim}H$ commuting variables. Our first main result shows that the whole of the above classical picture extends to the setting of Zhuang's theorem.

\begin{theorem} \label{firstmain}
Let $H$ be a connected Hopf algebra of finite Gel'fand-Kirillov dimension $n$ over an algebraically closed field $k$ of characteristic $0$ and let $T$ be a left or right coideal subalgebra of $H$. Then
\begin{enumerate}
\item{($\mathrm{Masuoka}$, \cite{Masuoka}) $H$ is a free right and left $T$-module.}
\item{
With respect to the coradical filtration of $T$, the associated graded algebra $\gr T$ is a polynomial $k$-algebra in $m$ variables. }
\item  $\GK{T}=m \leq n$, and $m=n$ if and only if $T = H$.
\end{enumerate}
\end{theorem}

The theorem is given below as Lemma \ref{Tissiglan} and Theorem \ref{homological}. Just as with parts (1) and (2), so also (3) incorporates familiar classical phenomena: if $W \subset U$ is a strict inclusion of unipotent $k$-groups, then $\mathrm{dim}W < \mathrm{dim}U$; and a strict inclusion of Lie algebras of course implies a strict inequality of their dimensions.

A key point is that the polynomial algebra $\mathrm{gr}H$ above is itself a Hopf algebra, the comultiplication and multiplication of $H$ being ``lifts" of those of $\mathrm{gr}H$. In particular, $\mathrm{gr}H$ is \emph{connected}, and hence is the coordinate ring of a unipotent algebraic $k$-group $U$, as explained in $\S$\ref{commutative}. Moreover, $T$ is a lift of the homogeneous $U$-space with coordinate ring $\mathrm{gr}T$. The case where $H$ is cocommutative, that is, $H$ is isomorphic as a Hopf algebra to the enveloping algebra of an $n$-dimensional Lie algebra, is the case where $U$ is abelian, that is, $U \cong (k,+)^n$.

\subsection{}\label{intro3} Following \cite{LiuWu}, we call a left or right coideal subalgebra $T$ of a Hopf algebra $H$ a \emph{quantum homogeneous space} if $H$ is faithfully flat as a left \emph{and} as a right $T$-module. Thus (1) of the theorem ensures that all coideal subalgebras of $H$ as in the theorem are quantum homogeneous spaces.

As is completely standard, a filtered-graded result such as Theorem \ref{firstmain} has important homological consequences. Thus we deduce:

\begin{corollary} \label{homo}Let $H$ and $T$ be as in Theorem \ref{firstmain}.
\begin{enumerate} \item{
$T$ is a noetherian domain of global dimension $m$, AS-Regular and GK-Cohen-Macaulay .}
\item{
$T$ is twisted Calabi-Yau of dimension $m$.}
\end{enumerate}
\end{corollary}

For unexplained terminology in the above, see the references and definitions in $\S$\ref{basic}, $\S$\ref{CY}, where these results are proved. The twisting automorphism in (2) is discussed in the next subsection.

\subsection{The antipode} \label{anti} As already mentioned, one of the pillars on which our work stands is the paper of Masuoka \cite{Masuoka}. The relevant specialisation of the main result of that paper is stated here as Proposition \ref{quantmasuoka}, a central feature being the existence of a bijection between the left coideal subalgebras of $H$ and the quotient right $H$-module coalgebras of $H$. Using this bijection and a well-known lemma due to Koppinen \cite{Ko}, we deduce the following, where (1) is well-known, and valid in a much wider setting (see Lemma \ref{general}), and the remaining parts are proved in Theorem \ref{antisquare} and Proposition \ref{conak}.

\begin{theorem}\label{antipodethrm} Let $H$ and $T$ be as in Theorem \ref{firstmain}, and let $S$ denote the antipode of $H$.
\begin{enumerate}
\item The map $T \mapsto S(T)$ gives a bijection between the sets of right and left coideal subalgebras of $H$.
\item $S^2(T) = T$, and either $(S^2)_{|T} = \mathrm{id}$ or $|(S^2)_{|T}| = \infty$.
\item Suppose $T$ is a right quantum homogeneous space. Then the Nakayama automorphism of $T$ is $S^2 \circ \tau^{\ell}_{\chi}$, where $\chi$ is the character of the right structure of the left integral of $T$. There is an analogous formula which applies when $T$ is a left quantum homogeneous space.
\end{enumerate}
\end{theorem}

Both possibilities can occur in (2): for $T = H$, the smallest example with $|S|$ infinite occurs at Gel'fand-Kirillov dimension 3, namely the infinite family of examples $B(\lambda)$ found in \cite{Zhuang} and recalled here in $\S$\ref{Blambda}.  For $T \neq H$, $S^2$ can have infinite order already at dimension 2, as we show (for a coideal subalgebra of $B(\lambda)$) in (\ref{infinity}) in $\S$\ref{smallHopf}. For the unexplained terminology in (3), see $\S$\ref{CY}. The determination of the Nakayama automorphism in (3) depends crucially on earlier work of Kraehmer \cite{Uli} and of Liu and Wu \cite{LiuWu}, \cite{Rigid}.

\subsection{The signature and the lantern} \label{siggy} Let $H$ and $T$ be as in Theorem \ref{firstmain}. Given that  $\mathrm{gr}H$ and $\mathrm{gr}T$ are graded polynomial algebras, their homogeneous generators have specific degrees whose multisets of values constitute invariants $\sigma(T)$ and $\sigma(H)$ of $T$ [resp. $H$]. We call $\sigma(T)$ the \emph{signature} of $T$, and write $\sigma (T) = (e_1^{(r_1)}, \ldots , e_s^{(r_s)})$, where $e_i$ and $r_i$ are positive integers with $e_1 < e_2 < \cdots < e_s$, and the term $e_i^{(r_i)}$ indicates that the degree $e_i$ occurs $r_i$ times among the graded polynomial generators of $\mathrm{gr}T$. A closely related invariant is the \emph{lantern} $\mathcal{L}(T)$ of $T$, defined in \cite[Definition 1.2(d)]{WZZ4} for $H$ itself, extended here to a quantum homogeneous space $T$: namely, $\mathcal{L}(T)$ is the $k$-vector space of primitive elements of the graded dual $(\mathrm{gr}(T))^*$. The basic properties of these invariants are as follows.

\begin{theorem} \label{siglant} Let $H$ be a connected Hopf algebra of finite Gel'fand-Kirillov dimension $n$ over an algebraically closed field $k$ of characteristic $0$ and let $T$ be a left coideal subalgebra of $H$ with $\mathrm{GKdim}T = m > 0$. Let $\sigma (H) = (d_1^{(m_1)}, \dots , d_t^{(m_t)})$ and $\sigma (T) = (e_1^{(r_1)}, \ldots , e_s^{(r_s)})$. Then
\begin{enumerate}
\item $\sum_i d_i m_i = n$ and $\sum_j e_j r_j = m$.
\item $d_1 = e_1 = 1$; if $n \geq 2$, then $r_1 \geq 2$.
\item $\sigma (H) = (1^{(n)})$ if and only if, as a Hopf algebra, $H \cong U(\mathfrak{g})$, the enveloping algebra of its $n$-dimensional Lie algebra $\mathfrak{g}$ of primitive elements.
\item $\mathcal{L}(H) = \bigoplus_i \mathcal{L}(H)(d_i)$ is a positively graded Lie algebra of dimension $n$, with $\mathrm{dim}_k \mathcal{L}(H)(d_i) = m_i$.
\item $\mathcal{L}(H)$ is generated in degree 1 - that is, $\mathcal{L}(H) = \langle \mathcal{L}(H)(1) \rangle$.
\item If $i < t$ then $d_{i+1} = d_i + 1$.
\item $\mathcal{L}(T)$ is a graded quotient of $\mathcal{L}(H)$.
\item $\sigma (T)$ is a submultiset of $\sigma (H)$.
\end{enumerate}
\end{theorem}

Parts (2),(3),(4), and (5) of this theorem were proved in \cite{WZZ4}, but they are proved again here in the course of proving the remaining parts, in $\S$5. Finite dimensional positively graded (and therefore nilpotent) Lie algebras which are generated in degree 1, as $\mathcal{L}(H)$ is by (4) and (5) of the theorem, are called \emph{Carnot Lie algebras} in the literature; see for example \cite{Corn}.

We view one important function of these equivalent invariants as being to provide a framework for future work on connected Hopf algebras. We discuss this aspect further in $\S$\ref{open}.

\subsection{Examples}\label{examples} In $\S$3 we describe what is known about various particular subclasses of quantum homogeneous spaces $T$ of connected Hopf algebras $H$ of Hopf algebras of finite Gel'fand-Kirillov dimension. Thus, we discusses the cases where $H$ or $T$ is commutative; where $H$ or simply $T$ is cocommutative; where $\mathrm{GKdim}T \leq 2$; and where $\mathrm{GKdim}H \leq 3$. Particularly noteworthy is the fact, recorded in Proposition \ref{Jordanprop} in $\S$\ref{tiny}, that
\begin{align*} &\textit{the Jordan plane } J = k\langle X,Y: [X,Y] = Y^2\rangle \textit{ is a coideal subalgebra }\\ &\textit{     of a connected Hopf algebra } H \textit{ with } \mathrm{GKdim}H = 3.\end{align*}

 It is not hard to see that $J$ cannot be supplied with a comultiplication with respect to which it is a Hopf algebra.

\subsection{Layout and notation}\label{notation} Some basic definitions, the keystone theorems of Masuoka and Zhuang, and the relevant properties of the antipode, are given in $\S$2. Examples are discussed in $\S$3; particularly crucial here is the commutative case, since this yields the required properties of $\mathrm{gr}T$ and $\mathrm{gr}H$. The homological results are described and proved in $\S$4, and the signature and lantern are covered in $\S$5. $\S$6 contains a short discussion of future research directions.

Throughout let $k$ denote a base field, which we assume to be algebraically closed and of characteristic $0$, unless otherwise stated. All vector spaces and unadorned tensor products will be assumed to be over the base field $k$. For a Hopf algebra $H$, we use $\Delta$, $\epsilon$ and $S$ to denote the coproduct, counit and antipode respectively. The kernel of the counit will be denoted as $H^{+}$. Throughout this paper, the bijectivity of the antipode $S$ is incorporated as part of the definition of a Hopf algebra; in practice this makes no difference to the main results, since bijectivity of $S$ always holds when $H$ is connected by \cite[\textrm{Proposition 1.2}]{Masuoka}.

\section{Basic definitions}\label{preliminaries}

\subsection{The Coradical Filtration}\label{coradfilt}
 Let $H$ be a Hopf algebra. Recall that an element $x\in H$ is $\emph{primitive}$ if $\Delta(x)=1\otimes x+x\otimes 1.$ The subspace of primitive elements of $H$, denoted by $P(H)$, is called the $\emph{primitive space}$ of $H$; $P(H)$ is a Lie subalgebra of $H$ with respect to the commutator bracket $[x,y]=xy-yx$. Let $H_{0}$ denote the \emph{coradical} of $H$, that is, the sum of all simple coalgebras of $H$, and define inductively the ascending chain of subcoalgebras of $H$,
\[
H_{n}:=\Delta^{-1}(H\otimes H_{n-1}+ H_{0}\otimes H).
\]
Then $\{H_{n}\}_{n=0}^{\infty}$ is the \emph{coradical filtration of} $H$.

Suppose now that $H$ is \emph{connected}, that is, $H_{0}=k$. Then $H_1 = k \oplus P(H)$, and, by \cite[\textrm{Lemma 5.2.8}]{Mont} $\{H_{n}\}_{n=0}^{\infty}$ forms an algebra as well as a coalgebra filtration of $H$, such that $S(H_{i})\subseteq H_{i}$ for $i\geq 0$; that is, it is a \emph{Hopf filtration}. Hence, the associated graded coalgebra of $H$ with respect to the coradical filtration,
$$\gr{H}=\bigoplus_{i=0}^{\infty} H(i):=\bigoplus_{i=0}^{\infty}H_{i}/H_{i-1}; \quad H_{-1}:=\{0\}.$$
forms a graded connected Hopf algebra. Furthermore, by a result of Sweedler (\cite[\S 11.2]{Sweedler}) with several subsequent rediscoveries, when $H$ is a connected Hopf algebra (over $\emph{any}$ field), $\gr H$ is commutative.

As noted in \cite[$\S$1]{WZZ4} and \cite[\textrm{Remark 5.5}]{Zhuang}, if $H$ is connected and $\operatorname{dim}_{k}(P(H))<\infty$ (which is the case when $\GK{H}<\infty$ since $U(P(H))\subseteq H$), then $\operatorname{dim}H(i)<\infty$ for all $i\geq0$. Thus, by \cite[Theorem 6.9]{Zhuang} (which we restate as Theorem \ref{beautiful} below), when $H$ is connected of finite GK-dimension,  $\gr H$ is a graded polynomial algebra in finitely many homogeneous variables. Moreover, $\gr H$ is a $\emph{coradically graded coalgebra} $ \cite[Definition 1.13]{And} - that is, for $n\geq 0$, $(\gr H)_{n}=\bigoplus_{i=0}^{n}H(i)$. Summarising the above discussion, we obtain

\begin{theorem}\label{beautiful}($\mathrm{Zhuang}$) \cite[Theorem 6.9]{Zhuang})
Assume that $k$ is an algebraically closed field of characteristic $0$ and let $H$ be a connected Hopf $k$-algebra, with associated notation as introduced above. Then the following statements are equivalent:
\begin{enumerate}
\item{$\operatorname{GKdim}{H}<\infty$.}
\item{$\operatorname{GKdim}{\operatorname{gr}H}<\infty$.}
\item{$\operatorname{gr}H$ is finitely generated.}
\item{$\operatorname{gr}H$ is isomorphic as an algebra to the polynomial $k$-algebra in $\ell$ variables for some $\ell\geq{0}$.}
\end{enumerate}
In this case, $\operatorname{GKdim}{H}=\operatorname{GKdim}{\operatorname{gr}H} = \ell$. Moreover,
$\gr H$ is connected both as a Hopf algebra and as a graded algebra.
\end{theorem}

Whenever we speak of the associated graded algebra of a connected Hopf algebra, we shall mean with respect to its coradical filtration. Similarly, the \emph{degree} of an element of a Hopf algebra will always be understood to be with respect to the coradical filtration, unless otherwise stated.

\subsection{Coideal subalgebras and quantum homogeneous spaces} \label{coideal} Let $H$ be a Hopf algebra. A subalgebra $T$ of $H$ is a \emph{left coideal subalgebra} if
$$\Delta(T)\subseteq H\otimes T.$$ Similarly we say that $T$ is a \emph{right coideal subalgebra} if $\Delta(T)\subseteq T\otimes H.$ The  \emph{coradical filtration} $\mathcal{T} := \{T_{n}\}_{n\geq 0}$ of a left coideal subalgebra $T$ of $H$ is defined to be  $$T_{n}:=T\cap H_{n}.$$ For each $n\geq 0$, $\Delta(T_{n})\subset H\otimes T_{n},$ and so the associated graded space
\[
\operatorname{gr}{T}:=\bigoplus_{i=0}^{\infty} T(i)\subseteq \gr{H}; \hspace{0.1in} T(i):=((T\cap H_{i})+H_{i-1})/H_{i-1}.
\]
satisfies the condition
\[
\Delta(\gr{T})\subset \gr{H}\otimes \gr{T}.
\]
Summing up the above we get the following.

\begin{lemma}\label{gradedcoideal} Let $H$ be a connected Hopf $k$-algebra of finite GK-dimension, and let $T$ be a left coideal subalgebra of $H$. Then $\gr{T}$ is a left coideal subalgebra of  the commutative graded connected Hopf algebra $\gr{H}$.
\end{lemma}

There is some inconsistency in the literature as to the precise definition of a quantum homogeneous space. For example, Kraehmer in \cite{Uli} defines a left quantum homogeneous space to be a left coideal subalgebra $C$ of a Hopf algebra $H$ such that $H$ is a faithfully flat left $C$-module. We adopt in this paper the formally more restrictive definition used in \cite{LiuWu}: a \emph{left quantum homogeneous space} of the Hopf algebra $H$ with a bijective antipode $S$ is a left coideal subalgebra $T$ of $H$ such that $H$ is a faithfully flat left \emph{and} right $T$-module. In fact, for the connected Hopf algebras which are the object of study in this paper, the distinction is irrelevant, as shown by the result below.

Given a quotient right $H$-module coalgebra $\pi (H)$ of $H$, define the \emph{right coinvariants} of $\pi$,
$$ H^{\operatorname{co}\pi} := \{ h \in H : \sum h_{1} \otimes \pi(h_{2}) = h \otimes \pi(1) \}. $$
Then $ H^{\operatorname{co}\pi}$ is a left coideal subalgebra of $H$. Analogously, the \emph{left coinvariants} of $\pi$ are
$$ ^{\operatorname{co}\pi}H := \{ h \in H : \sum \pi(h_{1}) \otimes h_{2} =  \pi(1)\otimes h \}, $$
a right coideal subalgebra of $H$.

\begin{proposition}\label{quantmasuoka}($\mathrm{Masuoka}$, \cite{Masuoka})
Let $H$ be a connected Hopf algebra and $T\subseteq H$ a subalgebra.
\begin{enumerate}
\item The antipode $S$ of $H$ is bijective.
\item $T$ is a left coideal subalgebra  if and only if it is a left quantum homogeneous space. In this case, $H$ is a free left and right $T$-module.

\item The correspondences $T \mapsto \{\pi_T: H \longrightarrow H/T^+ H \}$ and $\pi \mapsto H^{\operatorname{co}\pi}$ give a bijective correspondence between the left quantum homogeneous spaces of $H$ and the quotient right $H$-module coalgebras of $H$.
\item There is a parallel version of (3) for right coideal subalgebras of $H$.

\end{enumerate}
\end{proposition}
\begin{proof}
(1) This is \cite[Proposition 1.2(1)]{Masuoka}.

(2) Suppose $T$ is a left coideal subalgebra of the connected Hopf algebra $H$. Since $T_{0}=T\cap H_{0}=k$, $S(T_{0})=T_{0}$, and so, by \cite[\textrm{Theorem 1.3(1)}]{Masuoka}, $H$ is a left and right faithfully flat $T$-module. The last sentence is a special case of \cite[Proposition 1.4]{Masuoka}.

(3), (4): Since $S(T_0) = T_0$ as in (2), this is \cite[Theorem 1.3(3)]{Masuoka} and the version with right and left swapped.
\end{proof}

Note that, notwithstanding the above result, it is certainly not true that a left coideal subalgebra $T$ of an arbitrary Hopf algebra $H$ is a quantum homogeneous space. For example, let $H=k\langle x \rangle $, the group algebra of the infinite cyclic group. Then $T:=k[x]$ is both a left and right coideal subalgebra of $H$, but $\emph{not}$ a Hopf subalgebra, since $S(T)\nsubseteq T$; and $H$ is $\emph{not}$ a right or left faithfully flat $T$-module.
\\

The following (presumably well-known) lemma will be used later when classifying quantum homogeneous spaces of small GK-dimension.

\begin{lemma}\label{primitivecoideal}
Let $T$ be a left coideal subalgebra of a Hopf algebra $H$. Suppose $T \nsubseteq H_{0}$. Then
\[
T_{1}:=T\cap{H_{1}}\neq{T_{0}}.
\]
In particular, if $H$ is connected then
 $$P(T):=T\cap P(H)\neq k.$$
\end{lemma}
\begin{proof}
Let $\widehat{T}$ be a finite dimensional left subcomodule of $T$ with $\widehat{T} \nsubseteq H_{0}$; this exists by \cite[\textrm{Theorem 5.1.1}]{Mont}. Let $P$ denote the subcoalgebra of $H$ generated by $\widehat{T}$; $P$ is finite dimensional by \cite[\textrm{Theorem 5.1.1}]{Mont}. Then the finite dimensional algebra $P^{*}$ acts on $P$ via the right hit action,
$$
\leftharpoonup: P\otimes {P^{*}}\rightarrow P \;:\; p\otimes f \mapsto p\leftharpoonup f =\sum f(p_{1})p_{2}.
$$
where $\Delta(p)=\sum p_{1}\otimes p_{2}$. Since $\widehat{T}$ is a left subcomodule of $T$, if $t\in\widehat{T}$ then  $\Delta(t)=\sum t_{1}\otimes t_{2}\in P\otimes \widehat{T}$. Thus $\sum f(t_{1})t_{2}\in \widehat{T}$ for all $t\in\widehat{T}$, so $\widehat{T}$ is a $P^{*}$-submodule of $P$.

As is well known (see for example \cite[$\S$5.2]{Mont}), the terms $P_{i}=P\cap{H_{i}}$ of the coradical filtration of $P$ are precisely the elements annihilated by the powers $J(P^{*})^{i+1}$ of the Jacobson radical $J(P^{*})$ of $P^{*}$. Now $\widehat{T}\leftharpoonup J(P^{*})\neq 0$ since $\widehat{T} \not\subset H_0$. Since $P^*$ is a finite dimensional algebra,
$$
\text{Ann}_{\widehat{T}}(J(P^{*})) \subsetneqq \text{Ann}_{\widehat{T}}(J(P^{*})^{2}),
$$
as required.
\end{proof}

\subsection{Quantum homogeneous spaces and the antipode}\label{antipode}
Recall the following well-known and easy facts, where $k$ can be any field.
\begin{lemma}\label{leftright}

Let $H$ be a Hopf algebra with a bijective antipode.

\begin{enumerate}
\item The map $C \mapsto S(C)$ gives a bijection between the sets of left and right coideal subalgebras of $H$. This bijection restricts to a bijection between the left and right quantum homogeneous spaces of $H$.
\item If $C$ is a left or right coideal subalgebra, then $S(C) = C$ if and only if $C$ is a Hopf subalgebra of $H$.
\end{enumerate}
\end{lemma}
\begin{proof}
(1) Let $C$ be a left coideal subalgebra of $H$ and $x\in C$. Then
\[
\Delta(S(x))=\tau\circ(S\otimes S)\circ \Delta(x)=\sum S(x_{2})\otimes S(x_{1})\in S(C)\otimes H.
\]
That is, $S(C)$ is a right coideal subalgebra. Applying $S^{-1}$ shows that the correspondence is bijective. Clearly the correspondence preserves any flatness properties.

(2) This follows immediately from (1).
\end{proof}

In the following lemma, we only need to assume that $k$ does not have characteristic 2.

\begin{lemma}\label{general}
Let $H$ be a connected Hopf $k$-algebra.
\begin{enumerate}
\item   There exists a $k$-basis $\mathcal{B}$ of $H$ such that $\mathcal{B} \cap H_n$ spans $H_n$ for all $n$, and, for all $n$ and all $b \in \mathcal{B} \cap H_n$, there exists $r_b \in H_{n-1}$ such that
$S(b)=\pm b+r_b.$
\item Let $h\in H_{n}$. There exists $r \in H_{n-1}$ such that $S^2(h)= h+r.$
\end{enumerate}
\end{lemma}
\begin{proof}
(1) First, $S$ induces the antipode $\overline{S}$ of $\mathrm{gr}H$. Since $\mathrm{gr}H$ is a commutative Hopf algebra by \cite[\S 11.2]{Sweedler}, $\overline{S}$ has order 2 by \cite[\textrm{Corollary 1.5.12}]{Mont}. Thus the result follows by constructing $\mathcal{B} \cap H_n$ by induction on $n$, the new elements at stage $n$ being lifts of a basis of $\overline{S}$-eigenvectors of $H(n)$.

(2) This follows immediately from (1).
\end{proof}

\begin{theorem}\label{antisquare}

Let $H$ be a connected Hopf $k$-algebra, with $k$ a field of characteristic 0, and let $T$ be a (left or right) quantum homogeneous space in $H$.
\begin{enumerate}
\item $S^2 (T) = T$.
 \item Either $(S^{2})_{|T} =\operatorname{id}$ or $|(S^2)_{|T}|=\operatorname{\infty}$.
\end{enumerate}
\end{theorem}
\begin{proof}
(1) Suppose for definiteness that $T$ is a left coideal subalgebra of $H$. Under the bijective correspondence of Proposition \ref{quantmasuoka}(2), $T\leftrightarrow H/T^+ H$, and the right coideal subalgebra $S(T)$ corresponds to $H/H(S(T)^+)$. But
$$ H(S(T)^+) = HS(T^+) = S(T^+ H) = HT^+,$$
where the final equality is Koppinen's lemma, \cite[\textrm{Lemma 3.1}]{Ko}. Thus, applying $S$ now to the pair $\{ S(T), HS(T)^+ \}$, it follows that the left coideal subalgebra $S^2(T)$ is paired with $H/S(HT^+) = H/T^+H$. By the bijectivity of the correspondence of Proposition \ref{quantmasuoka}(3), we must have $S^2(T) = T$ as required.

(2) Suppose that $(S^{2})_{|T} \neq \operatorname{id}$, and choose $h \in T_n$ with $n$ minimal such that $S^2(h) \neq h$. By Lemma \ref{general}(2) there exists $0 \neq r \in H_{n-1}$ such that $S^2(h) = h + r$. By the first part of the theorem, $r \in T$. Hence, by choice of $n$, $S^2(r) = r$, so that, for all $m \geq 1$,
$$ S^{2m}(h) = h + mr.$$
As $k$ has characteristic 0, $S^{2m}(h) \neq h$ for every $m \geq 1$, as required.
\end{proof}

Note that part (1) of Theorem \ref{antisquare} holds, with the same proof, for any \emph{pointed} Hopf algebra $H$ and left coideal subalgebra $T\subset H$ with $S(T_{0})=T_{0}$. This is because Proposition \ref{quantmasuoka} is valid in this wider context. But things go wrong with part (2): consider for instance Sweedler's 4-dimensional pointed Hopf algebra \cite[\textrm{Example 1.5.6}]{Mont},
\[
K:=k\langle g, x, : g^{2}=1, x^{2}=0, xg=-gx\rangle.
\]
with $\Delta(g)=g\otimes g,\, \Delta(x)=x\otimes 1+g\otimes x, \,\epsilon(g)=1, \, \epsilon(x)=0, \, S(g)=g$ and $S(x)=-gx$. Here, $|S|=4$.

In part (2) of the theorem, both alternatives can occur. Of course $S^2 = \mathrm{Id}$ when $H$ is commutative or cocommutative, \cite[\textrm{Corollary 1.5.12}]{Mont}. On the other hand, as we illustrate in $\S$\ref{smallHopf}, there are connected Hopf algebras of GK-dimension 3 containing quantum homogeneous spaces $T$ of GK-dimension 2 with $|S^{2}_{|T}| = \infty$.

\bigskip

\section{Examples of quantum homogeneous spaces of connected Hopf algebras}


\subsection{Coideals in commutative connected Hopf algebras}\label{commutative}
Recall that $k$ is algebraically closed of characteristic 0, so that $H$ is a commutative affine Hopf $k$-algebra if and only if $H = \mathcal{O}(G)$ for an affine algebraic $k$-group $G$, \cite[\S 1.4]{Wa}.

\begin{theorem} \label{laz} Let $H = \mathcal{O}(G)$ as above, with $\mathrm{dim}G = n.$ Then the following are equivalent:
\begin{enumerate}
\item $H \cong k[X_1, \ldots , X_n]$ as algebras;
\item $G$ is unipotent;
\item $H$ is connected.
\end{enumerate}
\end{theorem}
\begin{proof} (1)$\Rightarrow$(2): This is Lazard's theorem, \cite{Lazard}.

(2)$\Rightarrow$(1): \cite[\textrm{Theorem 8.3}]{Wa}.

(2)$\Leftrightarrow$(3): \cite[\textrm{Theorem 8.3}]{Wa}.
\end{proof}

Now we review the homogeneous $G$-spaces for $G$ as in the above theorem. $H^{\mathrm{co}\pi}$ and $^{\mathrm{co}\pi}H$ have been defined in \ref{coideal}

\begin{theorem} \label{unipotent} Let $G$ be an affine unipotent algebraic $k$-group of dimension $n$, and let $T$ be a closed subgroup of $G$ with $\mathrm{dim}T = m$. Let $H = \mathcal{O}(G)$ and let $I$ be the defining ideal of $T$, so $I$ is a Hopf ideal of $H$ with $H/I \cong \mathcal{O}(T)$. Let $\pi : H \rightarrow H/I$ be the canonical epimorphism.
\begin{enumerate}
\item Let $K = H^{\mathrm{co}\pi}$ and $W =\, ^{\mathrm{co}\pi}H$. Then $K$ [resp. $W$] is a left [resp.right] coideal subalgebra of $H$, with $W = S(K)$ and $K = S(W)$.
\item $K$ is a Hopf subalgebra of $H \Leftrightarrow K = W \Leftrightarrow T$ is normal in $G\Leftrightarrow K^+H$ is a conormal ideal of $H$.
\item $K$ and $W$ are polynomial subalgebras of $H$ in $n-m$ variables, with
$$ H = K[t_1, \ldots , t_m] = W[t_1, \ldots ,t_m] $$
for suitable elements $t_i$ of $H$, $1 \leq i \leq m$.
\item Every coideal subalgebra of $H$ arises as above. Namely, if $K$ is a left coideal subalgebra of $H$ of dimension $n-m$, then $K^+ H$ is a Hopf ideal of $H$, so $H = \mathcal{O}(T)$ for a closed subgroup $T$ of $G$ with $\mathrm{dim} T = m$, and $K = H^{\mathrm{co}\pi}$ where $\pi :H \rightarrow H/K^+H$.
\end{enumerate}
\end{theorem}
\begin{proof} That $K = S(W)$ is a special case of the proof of Theorem \ref{antisquare}(1). The rest of (1), and (4), are special cases of Proposition \ref{quantmasuoka}, although of course they were known much earlier.

For part (3), note first that $I = K^+ H = W^+ H$ by (4), and $H/I \cong k[z_, \ldots ,z_m]$ by Theorem \ref{laz}. That the extensions $K \subset H$ and $W \subset H$ split as stated now follows from basic commutative algebra, taking $t_1, \dots, t_m$ to be any lifts of a set of polynomial generators from $H/I$ to $H$. The structure of $K$ and $W$ is a theorem of Rosenlicht, \cite[Theorem 1]{Rosenlicht}, but can also be seen easily by exploiting the basic properties of unipotent groups as in the proof of Theorem \ref{flag} below.

For (2), see \cite[3.4.2 and 3.4.3]{Mont}.
\end{proof}

\begin{theorem} \label{flag}Continue with the notation and hypotheses of the previous theorem for $H$ and $G$.
\begin{enumerate}
\item Let $K$ be a left coideal subalgebra of $H$, of dimension $n-m$. Then there is a complete flag
\begin{equation}\label{flaglike} k = K_0 \subset K_1 \subset \cdots \subset K_{n-m} = K \subset \cdots  \subset K_n = G, \end{equation}
where the $K_j$ are left coideal subalgebras of $H$ with $\mathrm{dim} K_j = j$ for all $j = 0, \dots , n$.
\item If $K$ is a Hopf subalgebra of $H$ then the $K_j$ in (\ref{flag}) can be chosen to be Hopf subalgebras.
\item Let each of $K$ and $L$ be a left or right coideal subalgebra of $H$, with $L \subseteq K$. Then
$$ \GK K = \GK L \Leftrightarrow K = L. $$
\end{enumerate}
\end{theorem}
\begin{proof} (1) and (2): By Theorem \ref{unipotent}(4), there is a closed subgroup $T$ of $G$, with $\mathrm{dim}T = m$, whose defining ideal is $K^+H$, and $K = H^{\mathrm{co}\pi}$ where $\pi: H \rightarrow H/K^+ H.$ The theorem follows by dualising the following standard facts about a unipotent group $G$ in characteristic 0, see for example \cite[\S 17.5]{Humphreys}:
\begin{itemize}
\item $G$ has a finite chain of normal subgroups $1 = G_0 \subset G_1 \subset \cdots \subset G_n = G$ with $\mathrm{dim}G_i = i$ and $G_{i+1}/G_i \cong (k,+)$ for all $i$;
\item if $T$ is a closed subgroup of $G$ with $T \neq G$ then the normaliser $N_G(T)$ of $G$ is closed, and $\mathrm{dim} N_G(T) > \mathrm{dim}T$.
\end{itemize}

(3) The construction of $T$ from $K$, and of the analogous closed subgroup $B$ of $G$ from $L$, does not depend on whether $K$ and $L$ are right or left coideal subalgebras. In all cases $T \subseteq B$ and (3) follows from the second of the above bullet points.
\end{proof}

\subsection{Cocommutative coideals in connected Hopf algebras}\label{cocommutative}

Let $C$ be a cocommutative left quantum homogeneous space in a connected Hopf algebra $H$. Then cocommutativity ensures that $C$ is in fact a connected sub-bialgebra of $H$, and so a connected Hopf subalgebra, \cite[\textrm{Lemma 5.2.10}]{Mont}. Thus one can apply the Cartier-Gabri$\grave{\mathrm{e}}$l-Kostant theorem  \cite[\textrm{Theorem 5.6.5}]{Mont}, characterising cocommutative connected Hopf algebras over a field of characteristic 0; we find that
$$ C \textit{ is isomorphic as a Hopf algebra to } \mathcal{U}(\mathfrak{g}),$$
where $\mathfrak{g}=P(C)$ is the Lie algebra of primitive elements of $C$. Conversely, of course, each subalgebra of the Lie algebra $P(H)$ gives rise to a cocommutative coideal subalgebra of $H$.

\subsection{Commutative coideals in connected Hopf algebras}\label{commutative}

Let $C$ be a commutative left quantum homogeneous space in a connected Hopf algebra $H$ of finite GK-dimension $n$. Then $C$ has finite GK-dimension. By Theorem \ref{Tissiglan} below there exists some $m\in \mathbb{Z}_{\geq 0}$ such that $m \leq \mathrm{GKdim}\mathrm{gr}H = n$ and $\mathrm{gr}C \cong k[X_1, \ldots , X_m]$. Moreover, $\GK{C}=\GK{\gr C}=m$,  with $m=n$ if and only if $C = H$. In particular, $C$ is affine, generated by any choice of lifts of the graded generators $X_i$ to $C$. Thus $C$, being commutative, is isomorphic to a factor of the polynomial $k$-algebra $R$ in $m$ variables. However, proper factors of $R$ have GK-dimension strictly less than $m$, so $C \cong R$. We have proved:

\begin{proposition}\label{commute} Let $C$ be a commutative left quantum homogeneous space in a connected Hopf $k$-algebra $H$ with $\mathrm{GKdim}H = n < \infty.$ Then $C$ is a polynomial algebra in $m$ variables, where $m \leq n$ and $m=n$ if and only if $C = H$.
\end{proposition}

\subsection{Quantum homogeneous spaces of small Gel'fand-Kirillov dimension}\label{tiny}Let $C$ be a left quantum homogeneous space in a connected Hopf $k$-algebra $H$, with $\mathrm{GKdim}H < \infty.$ We note here that, in dimensions 0 and 1, only classical homogeneous spaces occur in such a Hopf algebra; but a non-classical example occurs already with Gel'fand-Kirillov dimension 2.

(1) Suppose $\mathrm{GKdim}(C) = 0.$ Since $k$ is algebraically closed of characteristic 0, $H$ is a domain by \cite[Theorem 6.6]{Zhuang}. Since $C$ is locally finite dimensional, $C = k$ is the only possibility.

(2) Suppose $\mathrm{GKdim}(C) = 1.$ By Lemma \ref{primitivecoideal}, there exists $0 \neq c \in C \cap P(H)$. Thus $C$ contains the Hopf subalgebra $k[c]$ of $H$. It is  an immediate consequence of Proposition $\ref{chain lemma}$ below that in fact $k[c] = C$, so that:
\begin{align*} &\textit{ the only quantum homogeneous space } C \textit{ with } \mathrm{GKdim}C = 1 \textit{ in a}\\ &\textit{connected Hopf }
\textit{algebra } H \textit{ with } \mathrm{GKdim}H < \infty
\textit{ is } C = k[c] \textit{ with } c \textit{ primitive.}
\end{align*}

(3) Suppose $\mathrm{GKdim}(C) = 2.$ Then, in addition to the classical examples, namely the enveloping algebras of the two $k$-Lie algebras of dimension 2 with the usual cocommutative coproduct, one finds that the Jordan plane
\begin{equation}\label{Jordan} J = k\langle X,Y : [X,Y] = Y^2 \rangle \end{equation}
occurs as a coideal subalgebra of a connected Hopf algebra of Gel'fand-Kirillov dimension 3. The details are given in \S \ref{smallHopf} below. In particular,

\begin{proposition} \label{Jordanprop} There exists a quantum homogeneous space of a connected Hopf algebra of finite GK-dimension whose underlying algebra does not admit any structure as Hopf algebra. \end{proposition}

To prove this, one calculates that $\langle [J,J] \rangle = Y^2 J$, so that the abelianisation $J^{\mathrm{ab}}$ of $J$ is $k[X,Y]/\langle Y^2\rangle$. Since, as is easily checked, the abelianisation of a Hopf algebra is always a Hopf algebra, and commutative Hopf algebras in characteristic 0 are semiprime \cite[\textrm{Theorem 11.4}]{Wa}, the claim follows. This suggests an obvious project: classify the connected quantum homogeneous spaces of Gel'fand-Kirillov dimension 2. This is listed below as Question \ref{q2}.

\subsection{Connected Hopf algebras of small Gel'fand-Kirillov dimension}\label{smallHopf} Let $k$ as usual be algebraically closed of characteristic 0. Zhuang  in \cite[Examples 7.1, 7.2, Proposition 7.6, Theorem 7.8]{Zhuang} classified all the connected Hopf $k$-algebras $H$ with $\mathrm{GKdim}(H) \leq 3$. In dimension at most 2 one obtains only the enveloping algebras of the Lie algebras $\mathfrak{g}$ with $\mathrm{dim}_k(\mathfrak{g}) \leq 2$, with their standard cocommutative coproducts.

However, for Gel'fand-Kirillov dimension 3, in addition to the cocommutative examples, the classification is completed by two infinite series of (generically) noncommutative, noncocommutative connected Hopf $k$-algebras $A(\lambda, \mu, \alpha)$ and $B(\lambda)$. We record here the quantum homogeneous spaces of $B(\lambda)$; the classification for the algebras $A(\lambda, \mu, \alpha)$ follows a similar pattern. We shall return to the family $B(\lambda)$ again later to illustrate our results - see Example (\ref{lambdasig}).

Let $\lambda \in k$ and define
$$
B(\lambda) = k \langle X,Y,Z : [X,Y]=Y, [Z,X]=-Z+\lambda{Y}, [Z,Y]=\frac{1}{2}{Y}^{2}\rangle .
$$
Define $\Delta:B\rightarrow B\otimes B$, $S:B\rightarrow B$ and $\epsilon:B\rightarrow k$ by setting
\[
X,Y \in P(B(\lambda)); \quad  \Delta(Z)=1\otimes{Z}+X\otimes{Y}+Z\otimes{1}.
\]
and
\[
\epsilon(X)=\epsilon(Y)=\epsilon(Z)=0; \quad S(X)=-X, S(Y)=-Y, S(Z)=-Z+XY.
\]
One checks routinely that these definitions extend to yield algebra homomorphisms $\Delta$ and $\epsilon$ and an algebra antihomomorphism $S$. In \cite[\textrm{Example 7.2}]{Zhuang} it is proved that $(B, \Delta, \epsilon, S)$ is a connected Hopf algebra of GK-dimension 3. An easy calculation shows that  $S^{m}(Z)\neq Z$ for any $m> 0$, hence
\begin{equation}\label{infinity} B(\lambda) \textit{ is a connected Hopf algebra whose antipode has infinite order.}
\end{equation}
With $W := Z - \frac{1}{2}YX$, a simple calculation shows that, \emph{as an algebra}, $B(\lambda)$ is isomorphic to the enveloping algebra of the soluble $k$-Lie algebra $ \mathfrak{g}$ with basis $X,Y,W$ and relations
$$ [X,Y] = Y,  \quad [W,Y] = 0,\quad [X,W] = W - \lambda Y. $$

 To fix notation we recall the following:

\begin{definition}\label{Ore}
Let $R$ be a ring, $\sigma$ a ring automorphism of $R$, and $\partial$ a $\sigma$-derivation of $R$ (that is, an additive map $\partial:R\rightarrow R$ such that $\partial(rs)=\sigma(r)\partial(s)+\partial(r)s$ for all $s, r\in R$). We write $A=R[z; \sigma, \partial]$ and say that $A$ is an $\emph{Ore extension}$ of $R$ provided
\begin{enumerate}
\item{
$A$ is a ring, containing $R$ as a subring.}
\item{
$z$ is an element of $A$.}
\item{
$A$ is a free left $R$-module with basis $\{1, z, z^{2},\ldots\}$.}
\item{
$zr=\sigma(r)z+\partial(r)$ for all $r\in R$.}
\end{enumerate}
\end{definition}

By straightforward calculations, which we leave to the interested reader, one can verify the following:

\begin{proposition}\label{Blambda}
The following is a complete list of the non-trivial proper  coideal subalgebras of $B(\lambda)$.
\begin{enumerate}
\item Set $\mathfrak{g}_{\alpha}$ to be the 1-dimensional Lie algebra in $B(\lambda)$ with basis $\{X+\alpha Y\}$ ($\alpha\in{k}$), and $\mathfrak{g}_{\infty}$ to be the Lie algebra with basis $\{Y\}$. Then $\{U(\mathfrak{g}_{\alpha}):\alpha\in k\cup \{\infty\}\}$ is the complete set of Hopf subalgebras (and also left or right coideal subalgebras) in $B(\lambda)$ of Gel'fand-Kirillov dimension 1.

 \item Let $\delta_{\beta},\delta \in \text{Der}_{k}(U(\mathfrak{g}_{\infty}))$ be given by
 $$\delta_{\beta}(Y)=Y+\frac{\beta}{2}Y^{2} \textit{ and }\delta(Y)=\frac{1}{2}Y^{2}.$$
 Then $$L_{\beta} := U(\mathfrak{g}_{\infty})[X+\beta Z;\delta_{\beta}], \textit{ for } \beta\in{k}; \textit{  and  } L_{\infty} := U(\mathfrak{g}_{\infty})[Z;\delta] $$ are all the left coideal subalgebras of Gel'fand-Kirillov dimension 2 in $B(\lambda)$.

\item The right coideal subalgebras of Gel'fand-Kirillov dimension dimension 2 in $B(\lambda)$ are
$$R_{\beta} := U(\mathfrak{g}_{\infty})[X+\beta (Z-XY);\delta_{-\beta}], \textit{ for } \beta\in{k}\}$$ \textit{  and  }
$$ R_{\infty} := U(\mathfrak{g}_{\infty})[Z-XY; -\delta].$$
\item $S(L_{\beta}) = R_{\beta}$ for $\beta \in k \cup \{\infty\}$, and vice versa. Moreover, $L_0 = R_0$ is the only Hopf algebra in the lists (2) and (3), and is the only algebra which occurs in both lists.

\end{enumerate}
\end{proposition}

\section{Coideal Subalgebras of Connected Hopf Algebras}

\subsection{Associated graded algebras and Gel'fand-Kirillov dimension}\label{grade}

The starting point for almost everything which follows is:

\begin{lemma}\label{Tissiglan}
Let $H$ be a connected Hopf algebra of finite GK dimension $n$ and $T\subseteq{H}$ a left coideal subalgebra. Then $\operatorname{gr}T$ is a left coideal subalgebra of $\mathrm{gr}H$, and so  there exists some $m\leq n$ such that  $\gr T$ forms a graded polynomial algebra in $m$ variables, with $\mathrm{gr}H = \mathrm{gr}T[y_{m+1}, \ldots , y_n]$ for some elements $y_{m+1}, \ldots ,y_n$ of $\mathrm{gr}H$. Moreover, $\operatorname{GKdim}T=\operatorname{GKdim}\gr T=m\in\mathbb{Z}_{\geq 0}$.
\end{lemma}
\begin{proof}
Most of this is an immediate consequence of Lemma \ref{gradedcoideal} and Theorem \ref{unipotent}(3) and (4). For the final equality of gel'fand-Kirillov dimensions, use \cite[Proposition 8.6.5]{MR}.
\end{proof}

The first consequence of Lemma \ref{Tissiglan} is the quantum analogue of the fact (Theorem \ref{flag}(3)) that the normaliser $N_G(T)$ of a closed proper subgroup of a unipotent group in characteristic 0 has dimension strictly greater than that of $T$. It follows at once from Lemma \ref{Tissiglan} and Theorem \ref{flag}(3).

\begin{proposition}\label{chain lemma}
Let $H$ be a connected Hopf $k$-algebra of finite Gel'fand-Kirillov dimension. Let $T$ and $S$ be (left) coideal subalgebras of $H$ such that $S\subseteq{T}$. Then $S=T$ if and only if $\operatorname{GK dim}S=\operatorname{GK dim}T$.
\end{proposition}

\subsection{Basic properties of quantum homogeneous spaces of connected Hopf algebras}\label{basic}

Next, we show that the excellent homological properties enjoyed by connected Hopf algebras of finite GK-dimension extend to their quantum homogeneous spaces. The proof is a standard application of filtered-graded methods, closely following the case $T = H$ dealt with by Zhuang in \cite[Corollary 6.10]{Zhuang}. The terminology and notation used in the theorem is standard, and can be found for example in \cite{MR}, \cite{BrownGoodearl} or \cite{VO}. In particular, the \emph{(homological) grade} of a module $M$ over a $k$-algebra $R$ is $j_R(M) := \mathrm{inf}\{j: \mathrm{Ext}^j_R(M,R) \neq 0 \}.$

\begin{theorem}\label{homological}
Let $H$ be a connected Hopf $k$-algebra of finite Gel'fand-Kirillov dimension $n$. Let $T$ be a left coideal subalgebra of $H$ with $\operatorname{GKdim}T=m$.
\begin{enumerate}
\item{$T$ is a noetherian domain of Krull dimension at most $m$.}
\item{ $T$ is GK-Cohen-Macaulay.}
\item{$T$ is Auslander-regular of global dimension $m$.}
\item{$T$ is AS-regular of dimension $m$.}

\end{enumerate}
\end{theorem}
\begin{proof} (1) This follows from \cite[Proposition 1.6.6, Theorem 1.6.9 and Lemma 6.5.6]{MR}.

(2) Let $M$ be a finitely generated (left or right) $T$-module. Choose a good filtration of $M$, in the sense of \cite[\textrm{Definition 5.1}]{VO}. Then $\gr{M}$ is a finitely generated $\gr T$-module by \cite[\textrm{Lemma 5.4}]{VO}. By Lemma $\ref{Tissiglan}$, $\gr{T}$ is GK-Cohen-Macaulay. In particular,
\begin{equation}\label{GKM2}
j_{\gr T}(\gr M)+\GK{\gr M}=\GK{\gr T}.
\end{equation}
By Lemma $\ref{Tissiglan}$, $\GK{T}=\GK{\gr T} = m$. Since $\gr T$ is affine and $\gr M$ is a finitely generated $\gr T$-module, $\GK{M}=\GK{\gr M}$ by \cite[\textrm{Proposition 6.6}]{KL}. Finally, as in the proof of \cite[\textrm{Theorem 3.9}]{Bjork}, $j_{\gr{T}}(\gr M)=j_{T}(M)$. The result now follows from ($\ref{GKM2}$).

(3) By Lemma \ref{Tissiglan} and filtered-graded considerations \cite[Corollary 7.6.18]{MR}, $T$ has (right and left) global dimension at most $m$. Taking $M$ to be the trivial $T$-module $k$  in (2) yields $j_T(k)=m$. Hence the global dimension of $T$ is $m$.

The Hopf algebra $H$ is Auslander-Gorenstein by \cite[\textrm{Corollary 6.11}]{Zhuang}, and so, by \cite[\textrm{Proposition 2.3}]{LiuWu}, the left coideal subalgebra $T$ is too. Being Auslander-Gorenstein of finite global dimension, $T$ is by definition Auslander-regular.

(4) By part (2) and the fact that $\mathrm{gldim}T = m$, it remains only to check that the non-zero space $\operatorname{Ext}_{T}^{m}(k,T)$ satisfies
\begin{equation}\label{one} \mathrm{dim}_k \operatorname{Ext}_{T}^{m}(k,T) = 1. \end{equation}

In the sense of \cite[\textrm{Chapter 2.6}]{Bjork}, the filtration on the $T$-module $k$ (coming from the coradical filtration of $T$) is $\emph{good}$, \cite[8.6.3]{MR}. This yields good filtrations on the $T$-modules $\operatorname{Ext}_{T}^{j}(k,T)$ - let $\gr_{*}(\operatorname{Ext}_{T}^{j}(k,T))$ denote the associated graded $\gr{T}$-modules, for $j = 0, \ldots , m$. By \cite[\textrm{Proposition 6.10}]{Bjork}, $\gr_{*}(\operatorname{Ext}_{T}^{j}(k,T))$ is a sub factor of the $\gr{T}$-module $\operatorname{Ext}_{\gr{T}}^j(k, \gr{T})$. By Lemma \ref{Tissiglan} $\gr{T}$ is AS-regular of dimension $m$, so that $\operatorname{Ext}_{\gr{T}}^{m}(k, \gr{T}) = k$. Thus
$$ \mathrm{dim}_k \operatorname{Ext}_{T}^{m}(k,T) = \mathrm{dim}_k \operatorname{Ext}_{\gr{T}}^{m}(k, \gr{T}) \leq 1, $$
as required.
\end{proof}

The enveloping algebra $H = U(\mathfrak{g})$ of a finite-dimensional simple complex Lie algebra $\mathfrak{g}$ has Krull dimension $\mathrm{dim}_{\mathbb{C}}(\mathfrak{b})$, where $\mathfrak{b}$ is a Borel subalgebra of $\mathfrak{g}$, \cite{Lev}. Thus the inequality in Theorem \ref{homological}(1) is strict in general.

\subsection{The Calabi-Yau property for quantum homogeneous spaces} \label{CY} Let $A$ be a $k$-algebra. For a left $A^{e}=A\otimes A^{\operatorname{op}}$-module (that is, $A$-bimodule) $M$ and $k$-algebra endomorphisms $\nu, \sigma$ of $A$, denote by $^{\nu}M^{\sigma}$ the $A^{e}$-module whose underlying vector space is $M$, with $A^{e}$-action
\[
a\cdot m\cdot b= \nu(a)m \sigma(b).
\]
for $a,b \in A, m\in M$. If $\nu = \mathrm{Id}$, write $M^{\sigma}$ rather than $^{\mathrm{Id}}M^{\sigma}$.

\begin{definition}(\cite{KZhang})
An algebra $A$ is  \emph{$\nu$-twisted Calabi-Yau of dimension $d$} for a $k$-algebra automorphism $\nu$ of $A$ and an integer $d\geq{0}$ if
\begin{enumerate}
\item{
$A$ is \emph{homologically smooth}, that is, as an $A^{e}$-module, $A$ has a finitely generated projective resolution of finite length;}
\item{
$\operatorname{Ext}_{A^{e}}^{i}(A,A^{e})\cong\delta_{i,d}A^{\nu}$ as $A^{e}$-modules, where the $A^{e}$-module structure on the $\operatorname{Ext}$ group is induced by the right $A^{e}$-module structure of $A^{e}$.
}
\end{enumerate}

\noindent Then $\nu$ is uniquely determined up to an inner automorphism and is called the \emph{Nakayama automorphism} of $A$. Some authors use the term ``$\nu$-skew'' rather than ``$\nu$-twisted''. We omit the adjective ``twisted" if $\nu$ is inner.
\end{definition}

\begin{remark}\label{integral}(\cite{WuZhang})
Let $A$ be an AS-Gorenstein $k$-algebra of dimension $d$. Taking $k$ to be the \emph{left} $A$-module annihilated by the augmentation ideal $A^+$ of $A$, the one-dimensional space $\operatorname{Ext}_{A}^{d}({_{A}k},{_{A}A})$ is the \emph{left homological integral} of $A$, denoted $\int_{A}^{\ell}$. From its definition, it has an induced $A$-bimodule structure: the left $A$-action is induced by the trivial action on $k$, whereas the right $A$-module structure on $\int_{A}^{\ell}$ is induced from the right $A$-module structure on $A$. Thanks to the AS-Gorenstein hypothesis, this right $A$-module structure induces a character $\chi:A\rightarrow{k}$ such that
\[
f\cdot a=\chi(a)f.
\]
for all $f\in \int_{A}^{l}$ and $a \in A$.

\end{remark}

Recall that a Hopf algebra $A$ (with bijective antipode) which is noetherian and AS-Gorenstein of dimension $n$ is twisted Calabi-Yau of dimension $n$, \cite[\textrm{Proposition 4.5}]{KZhang}. In view of \cite[Corollary 6.10]{Zhuang}, the homological corollary of Theorem \ref{beautiful}, this applies in particular to a connected Hopf $k$-algebra $H$ of finite GK-dimension $n$. Moreover, also by \cite[\textrm{Proposition 4.5}]{KZhang}), the Nakayama automorphism $\nu$ of $H$ is given by
$$\nu=\tau_{\chi}^{\ell}\circ S^{2},$$
where $S$ denotes as usual the antipode of $H$ and $\tau_{\chi}^{\ell}$ denotes the left winding automorphism of the character $\chi:H\rightarrow k$ defined by $\int_{H}^{\ell}$ as in Remark \ref{integral}. That is, $\tau_{\chi}^{\ell}(h) = \sum \chi (h_1)h_2$ for $h \in H$.
\\

Naturally, one asks: does this generalise to a right quantum homogeneous space $T$ of a connected Hopf $k$-algebra $H$ of finite GK-dimension?
\\

 By Theorem \ref{homological}(4) such a right coideal subalgebra $T$ is AS-Gorenstein, and so has a left homological integral $\int_{T}^{\ell}$, with character $\chi:T\rightarrow k$. \emph{A priori}, the map $$\tau_{\chi}^{\ell}:T\rightarrow H : t \mapsto \sum \chi(t_1)t_2 $$
 while easily seen to be an algebra homomorphism, might not take values in $T$. But this \emph{is} in fact so for all such quantum homogeneous spaces $T$ in $H$, by \cite[\textrm{Lemma 3.9}]{LiuWu}, building on work of \cite{Uli}. Consequently, we obtain:

\begin{proposition} \label{conak}
Let $H$ be a connected Hopf $k$-algebra with finite Gel'fand-Kirillov dimension.
 \begin{enumerate}
 \item Let $T$ be a right coideal subalgebra of $H$ with $\operatorname{GKdim}{T}=m$. Then $T$ is twisted Calabi-Yau of dimension $m$. Retaining the notation introduced above, so in particular $\chi$ is the character of the right structure of the left integral of $T$, the Nakayama automorphism $\nu$ of $T$ is
 \[
\nu=S^{2}\circ \tau_{\chi}^{\ell}.
\]
\item The same conclusions apply to a left coideal subalgebra $T$ of $H$, with $\chi$ as in (1), with Nakayama automorphism $\nu$ of $T$
\[
\nu=S^{-2}\circ \tau_{\chi}^{r}.
\]
\end{enumerate}
\end{proposition}
\begin{proof} (1) That $T$ is homologically smooth follows from \cite[\textrm{Lemma 3.7}]{Rigid}. Given that $T$ is AS-regular by Theorem \ref{homological}(4), \cite[\textrm{Theorem 3.6}]{Rigid} implies that  $T$ is $\nu$-twisted Calabi-Yau , with $\nu=S^{2}\circ\tau_{\chi}^{\ell}$.

(2) Let $T$ be a left coideal subalgebra of $\mathbf{H} := (H, \mu, \Delta,S,\epsilon)$. By \cite[Lemma 1.5.11]{Mont}, $\mathbf{H'} := (H, \mu, \Delta^{op},S^{-1},\epsilon)$ is a Hopf algebra, clearly connected, and with the Gel'fand-Kirillov dimensions of $H$ and its subalgebras unchanged, since the algebra structure is the same. However, $T$ is now a \emph{right} coideal subalgebra of $\mathbf{H'}$, so part (1) can be applied to it. Hence, the Nakayama automorphism is as stated in (2), with a \emph{right} winding automorphism appearing now (with respect to $\Delta$) because the coproduct for $\mathbf{H}$ is $\Delta^{op}$.
\end{proof}

Here is a typical example to illustrate the proposition.

\begin{example}\label{Blambdanak} $\textbf{Nakayama automorphism of a coideal subalgebra of }  B(\lambda)$.

The notation is as introduced in $\S$\ref{smallHopf}. So the right coideal subalgebra we consider is $R_{\infty} :=k\langle Y, W\rangle \subseteq B(\lambda)$, where $W:=Z-XY$. As noted in Proposition \ref{Blambda}, $ [W,Y]=-\frac{1}{2}Y^{2}$. Then $R_{\infty}$  is a right coideal subalgebra of $B(\lambda)$ with $\mathrm{GKdim}R_{\infty} = 2$, and $R_{\infty}$ is isomorphic to the Jordan plane. By Proposition $\ref{conak}$, to compute the Nakayama automorphism of $R_{\infty}$ we must compute the right $R_{\infty}$-module structure of $\int^{l}_{R_{\infty}}=\operatorname{Ext}_{R_{\infty}}^{2}(k, R_{\infty})$.
\\

For an automorphism $\tau$ of $R_{\infty}$, call $b\in R_{\infty}$ \emph{$\tau$-normal} if $\tau(a)b=ba$ for all $a\in R_{\infty}$. Thus $Y$ is a $\sigma$-normal element of $R_{\infty}$, where $\sigma(Y) = Y$, $\sigma (W) = W+\frac{1}{2}{Y}.$  Set $\overline{R_{\infty}}=R_{\infty}/YR_{\infty}$.  Then $\overline{R_{\infty}} \cong k[W]$, so that
$$\int^{l}_{\overline{R_{\infty}}}=\operatorname{Ext}_{\overline{R_{\infty}}}^{1}(k, \overline{R_{\infty}})\cong \; ^1k^1.$$
Hence, by the noncommutative Rees Change of Rings theorem, \cite[\textrm{Lemma 6.6}]{KZhang}, there is an $R_{\infty}$-bimodule isomorphism
\begin{equation}\label{sigma}
\int_{R_{\infty}}^{l}\cong {\int_{\overline{R_{\infty}}}^{l}}^{\sigma^{-1}} \cong \; ^1 k^{\sigma^{-1}} \cong \; ^1 k^1,
\end{equation}
where the final isomorphism above holds since twisting by $\sigma^{-1}$ does not alter the structure of the trivial $R_{\infty}$-module. Therefore, by Proposition $\ref{conak}$, $$\nu_{R_{\infty}}=S^{2}_{|R_{\infty}},$$
so that $$\nu (Y) = Y \textit{ and } \nu (W) = W - Y.$$

The above calculation agrees with that carried out for the Jordan plane by other means in, for example \cite[$\S$4.2]{LiuWang}.
\end{example}

\begin{remark}\label{antiformula} If $T$ is a Hopf subalgebra of $\mathbf{H} = (H,\Delta,S,\epsilon)$, then both parts of the proposition apply to it. The Nakayama automorphism of a skew Calabi-Yau algebra is unique up to an inner automorphism of the algebra, but in this case $H$ and thus $T$ have no non-trivial inner automorphisms, thanks to the fact that the only units of connected Hopf $k$-algebras of finite Gel'fand-Kirillov dimension are in $k^*$, since this is true for their associated graded algebras by Zhuang's Theorem \ref{beautiful}. Thus, generalising \cite[4.6]{KZhang} for the particular case of connected algebras, we find that, for a Hopf subalgebra $T$ of the connected Hopf $k$-algebra $H$ of finite Gel'fand-Kirillov dimension,
\begin{equation} S_{|T}^4 = \tau_{-\chi}^{\ell}\circ\tau_{\chi}^r ,
\end{equation}
where $\chi$ is the character of the right structure on $\int_T^{\ell}$. We find this formula rather curious, given that there is no obvious relationship between $\int_T^{\ell}$ and $\int_H^{\ell}$.
\end{remark}

The following corollary of Proposition \ref{conak} ought to have a more direct proof.

\begin{corollary}\label{anticomco} Let $C$ be a commutative right or left coideal subalgebra of a connected Hopf $k$-algebra $H$ of finite Gel'fand-Kirillov dimension. Then $S^2_{|C} = \mathrm{Id}_C$.
\end{corollary}
\begin{proof} Apply Proposition \ref{conak} to $C$. Commutativity of $C$ ensures that both the character $\chi$ and the automorphism $\nu$ are trivial, meaning $\chi = \epsilon$ and $\nu = \mathrm{Id}_C$. Substituting these values in the formula for $\nu$ gives the desired conclusion.
\end{proof}

Note that the corresponding result to the above with ``cocommutative'' replacing ``commutative'' is rather trivial, since then $C$ is a cocommutative Hopf subalgebra of $H$, as was noted in $\S$\ref{cocommutative}.

\section{Invariants of Quantum Homogeneous Spaces}

\subsection{Preliminaries on gradings}\label{gradings} Let $A=\bigoplus_{i\geq0} A(i)$ be a connected $\mathbb{N}$-graded $k$-algebra with $\operatorname{dim}_{k}(A(i))<\infty$ for all $i\geq 0$. We write
\[ h_{A}(t)=\sum_{i=0}^{\infty}\operatorname{dim}_{k}A(i)t^{i}
\]
for the $\emph{Hilbert series}$ of $A$. We need the following well-known and easy lemma.

\begin{lemma}\label{tensor}
Let $A, B$ and $C$ be locally finite connected $\mathbb{N}$-graded algebras, with $A\cong B\otimes_k C$ as graded algebras. Then $$h_{A}(t)=h_{B}(t)h_{C}(t).$$
\end{lemma}

The next easy lemma is key to the definitions in this section.

\begin{lemma}\label{hilbkey} Let $R$ be a commutative connected $\mathbb{N}$-graded commutative polynomial $k$-algebra, with homogeneous polynomial generators $x_1, \ldots , x_n$. Let $\mathfrak{m} = \bigoplus_{i > 0} R(i) = \langle x_1, \ldots , x_n \rangle$ be the graded maximal ideal of $R$.
\begin{enumerate}
\item Homogeneous elements $y_1, \ldots , y_n$  form a set of polynomial generators of $R$ if and only if their images in $\mathfrak{m}/\mathfrak{m}^2$ form a $k$-basis for this space.
\item Let $C$ and $D$ be graded polynomial subalgebras of $R$ such that $C \subseteq D$ and $R = C[z_1, \ldots ,z_t] = D[w_1, \ldots , w_r]$; that is, $R$ is a polynomial algebra over $C$ and over $D$. Then there exist homogeneous elements $u_1, \ldots , u_n$ in $\mathfrak{m}$ such that
$$     C = k[u_1, \ldots ,u_{n-t}], \qquad D = k[u_1, \ldots ,u_{n-r}], \qquad R = k[u_1, \ldots ,u_{n}].$$
\item The multiset $\sigma(C)$ of degrees of a homogeneous set of polynomial generators of $C$ equals the multiset of degrees of a homogeneous basis of $\mathfrak{m}\cap C/(\mathfrak{m}\cap C)^2$, and hence is independent of the choice of such a generating set.
\item $\sigma (C) \subseteq \sigma (D)$, with equality if and only if $C = D$. Equivalently, the Hilbert polynomial $h_C (t)$, which equals $\prod_{d \in \sigma(C)}\frac{1}{(1-t^{d})}$, divides $h_D (t)$.
\end{enumerate}
\end{lemma}
\begin{proof}(1)$\Rightarrow$: This is trivial.

$\Leftarrow$: Let $y_1, \ldots , y_n$ be homogeneous elements of $R$ whose images \emph{modulo} $\mathfrak{m}^2$ form a $k$-basis for $\mathfrak{m}/\mathfrak{m}^2$. Define $A = k\langle y_1, \ldots , y_n\rangle$. Suppose that $A\subsetneqq R$, and let $s$ be minimal such that
$$ A(s) \subsetneqq R(s). $$
Since $R(s) \cap \mathfrak{m}^2$ is spanned by products of pairs of elements in $\{R(i) : i < s \}$,
\begin{equation}\label{hype} R(s) \cap \mathfrak{m}^2 = A(s) \cap \mathfrak{m}^2 \subseteq A(s). \end{equation}
Choose $x \in R(s) \setminus A(s)$, so $x  \notin \mathfrak{m}^2$. By hypothesis, there exist $\lambda_j \in k,$ $1 \leq j \leq n$, such that
$$ \widehat{x} := x - \sum_j \lambda_j y_j \in \mathfrak{m}^2.
$$
Clearly, for all $j$ with $\lambda_j \neq 0$, the degree of $y_j$ is $s$. But by (\ref{hype}) this forces $\widehat{x} \in A(s)$, so that $x \in A(s)$, contradicting the choice of $x$. Therefore we must have $A = R$, and finally considering Krull (or equivalently Gel'fand-Kirillov) dimension shows that $y_1, \ldots , y_n$ are polynomial generators of $A$.

(2) First, note that since $R = C \otimes_k k[z_1, \ldots , z_t]$, $\mathfrak{m}^2 \cap C = (\mathfrak{m} \cap C)^2$. Choose homogeneous elements $u_1, \ldots , u_{n-t}$ in $\mathfrak{m} \cap C$ whose images \emph{modulo} $\mathfrak{m}^2 \cap C = (\mathfrak{m} \cap C)^2$ form a $k$-basis for $\mathfrak{m} \cap C/\mathfrak{m}^2 \cap C$. Thus, by part (1), $C = k[u_1, \ldots , u_{n-t}]$. Then $\mathfrak{m} \cap C/\mathfrak{m}^2 \cap C$ embeds in $\mathfrak{m} \cap D/\mathfrak{m}^2 \cap D = \mathfrak{m} \cap D/(\mathfrak{m} \cap D)^2$, so we can extend $\{u_1, \ldots , u_{n-t}\}$ to a set of homogeneous elements $\{u_1, \ldots , u_{n-r}\}$ of $D$, of positive degree, whose images \emph{modulo} $\mathfrak{m}^2 \cap D$ provide a $k$-basis for $\mathfrak{m} \cap D/(\mathfrak{m} \cap D)^2$. By (1) again, $D = k[u_1, \ldots , u_{n-r}].$ A further repeat of the argument extends the set to homogeneous polynomial generators $\{u_1, \ldots , u_n\}$ of $R$.

(3) This is clear from (1) applied to $C$ rather than $R$, since the degrees and dimensions of the homogeneous components of $\mathfrak{m} \cap C/(\mathfrak{m} \cap C)^2$ are fixed.

(4) This is immediate from (2) and (3). The equivalent formulation in terms of Hilbert series follows from Lemma \ref{tensor}.
\end{proof}

\subsection{The signature and the lantern} \label{siglant}

Lemma \ref{hilbkey}(3) ensures that the following definitions make sense. That the previous parts of the definition apply to $T$ as in part (4) follows from Lemma \ref{Tissiglan}.

\begin{definition}\label{sigdef}
\begin{enumerate}
\item Let $R$ be a connected $\mathbb{N}$-graded polynomial algebra in $n$ variables, $n < \infty$. The $\emph{signature}$ of $R$, denoted by $\sigma(R)$, is the ordered $n$-tuple of degrees of the homogeneous generators.
\item Let $A$ be an $\mathbb{N}$-filtered algebra, with filtration $\mathcal{A}=\{A_{i}\}_{i\geq{0}}$, such that the associated graded algebra $\gr{A}=\bigoplus_{i}A_{i}/A_{i-1}$ is a connected $\mathbb{N}$-graded polynomial algebra in $n$ variables, $n < \infty$. The $\mathcal{A}-\emph{signature}$ of $A$, denoted by $\sigma_{\mathcal{A}}(A)$, is the signature of $\gr{A}$ in the sense of (1). Where no confusion is likely refer simply to the $\emph{signature}$ of $A$, denoted $\sigma(A)$.
\item With $A$, $\mathcal{A}$ and $\mathrm{gr}A$ as in (2), suppose that $\mathrm{gr}A$ has $m_i$ homogeneous generators of degree $d_{i}, \; 1 \leq i \leq t$, with $1 \leq d_{1} < \cdots < d_t$, so that $\sum_{i=1}^t m_i = n$. Then we write
    $$ \sigma(A) = (d_{1}^{(m_{1})},\ldots, d_{t}^{(m_{t})}).$$
    When  $m_{i}=1$,  the exponent $(m_{i})$ is omitted.
\item Let $k$ be an algebraically closed field of characteristic 0. Let $H$ be a connected Hopf $k$-algebra of finite Gel'fand-Kirillov-dimension, with a left coideal subalgebra $T$, with coradical filtration $\mathcal{T}$ as defined in $\S$\ref{coideal}. Then the $\emph{signature}$ of $T$, denoted $\sigma(T)$, is the $\mathcal{T}$-signature of $T$ as defined in (2).
\end{enumerate}
\end{definition}
Dualising the above definition, as follows, gives certain benefits, as we shall see. The definition of the \emph{lantern} of a connected Hopf $k$-algebra, and the key parts (1), (2), (4) and the corollary of Proposition \ref{litup}, are due to Wang, Zhang and Zhuang, \cite[Definition 1.2 and Lemma 1.3]{WZZ4}. Proofs are given again here for the reader's convenience, in the course of extending their definition.

\begin{definition}\label{lanterndef}
\begin{enumerate}
\item Let $R=\bigoplus_{i=0}^{\infty} R(i)$ denote a connected $\mathbb{N}$-graded polynomial algebra in $n$ variables, $n<\infty$. Let $\mathcal{D}_{R} := \bigoplus_{i=0}^{\infty} R(i)^*$ be the \emph{graded dual} of $R$, so $\mathcal{D}_R$ is a graded cocommutative coalgebra.
\item The \emph{lantern} $\mathfrak{L}(R)$ of $R$  is the space  of primitive elements of $\mathcal{D}_R$.  That is,
\[
\mathfrak{L}(R):=P(\mathcal{D}_{R}),
\]
a graded subcoalgebra of $\mathcal{D}_R$. Note that $\mathfrak{L}(R)$ is non-zero provided $R \neq k$, since then $k=\mathcal{D}_{R}(0)\subsetneqq \mathcal{D}_{R}$).
\item Let $A$ be an $\mathbb{N}$-filtered $k$-algebra, with filtration $\mathcal{A}=\{A_{n}\}_{n\geq 0}$, satisfying the hypotheses of Definition $\ref{sigdef}(2)$. Define the $\mathcal{A}-\emph{lantern}$ of  $A$, denoted by $\mathfrak{L}_{\mathcal{A}}(A)$, to be the lantern of the graded polynomial algebra $\gr{A}$. Where no confusion is likely, shorten notation to $\mathfrak{L}(A)$.
\item Let $k$ be an algebraically closed field of characteristic 0. Let $H$ be a connected Hopf $k$-algebra of finite Gel'fand-Kirillov-dimension, with a left coideal subalgebra $T$ with coradical filtration $\mathcal{T}$, as in $\S$\ref{coideal}. Then the $\emph{lantern}$ $\mathcal{L}(T)$ of $T$ is the $\mathcal{T}$-lantern of $T$ as defined in (3).
\end{enumerate}
\end{definition}

\begin{proposition}\label{litup} Let $H$ be a connected Hopf $k$-algebra of finite Gel'fand-Kirillov dimension $n$, and let $T$ be a left coideal subalgebra of $H$, with $\mathrm{GKdim}T = m.$ Let $\sigma(H) = (d_{1}^{(m_{1})},\ldots, d_{t}^{(m_{t})})$.
\begin{enumerate}
\item The graded dual $\mathcal{D}_H = (\mathrm{gr}H)^*$ is a sub-Hopf algebra of the finite dual $(\mathrm{gr}H)^{\circ}$; as such, it is isomorphic to the enveloping algebra $U(\mathcal{L}(H))$.
\item The lantern $\mathcal{L}(H)$ is a positively graded Lie algebra, with
$$ \mathcal{L}(H) = \bigoplus_{i=1}^t \mathcal{L}(H)(d_i), \;\; \mathrm{dim}_k\mathcal{L}(H)(d_i) = m_i, \; \; \mathrm{dim}_k \mathcal{L}(H) = n. $$
\item Let $W$ be the unipotent group of dimension $n$ such that $\mathrm{gr}H \cong \mathcal{O}(W).$ Then $\mathcal{L}(H)$ is the Lie algebra of $W$.
\item $\mathcal{L}(H)$ is generated in degree 1.
\item Let $Y_T$ be the closed subgroup of $W$ defined by the Hopf ideal $\mathrm{gr}T^+ \mathrm{gr}H$ of $\mathrm{gr}H$. Let $\mathfrak{y}_T$ be the Lie subalgebra of $\mathcal{L}(H)$ corresponding to $Y_T$. The graded dual $\mathcal{D}_T$ is a graded left $U(\mathcal{L}(H))$-module, isomorphic as such to $U(\mathcal{L}(H))/U(\mathcal{L}(H))\mathfrak{y}_T$.
\item The lantern $\mathcal{L}(T)$ is the graded quotient $\mathcal{L}(H)/\mathfrak{y}_T$.
\item The following are equivalent.
\begin{enumerate}
\item{
$\sigma(T)=(e_{1}^{(r_{1})},\ldots, e_{s}^{(r_{s})})$.}

\item{
$\mathrm{dim}_k\mathfrak{L}(T)(e_{i})=r_{i}$ for $i\geq 1$.}
\item{
$h_{\mathrm{gr}T}(t)=\prod_{i=1}^{s}\frac{1}{(1-t^{e_{i}})^{r_{i}}}$.}

\end{enumerate}

\end{enumerate}
\end{proposition}

\begin{proof}(1) Let $\mathfrak{m} = \oplus_{i \geq 1} H(i)$ be the graded maximal ideal of $\mathrm{gr}H$. By construction,
$$ \mathcal{D}_H = \{ f \in (\mathrm{gr}H)^{\circ} : f(\mathfrak{m}^j) = 0 \textit{ for } j \gg 0 \}. $$
By \cite[Proposition 9.2.5]{Mont}, $\mathcal{D}_H$ is a sub-Hopf algebra of $(\mathrm{gr}H)^{\circ}$; more precisely,
$$ \mathcal{D}_H = U(P(\mathcal{D}_H)),$$
the enveloping algebra of the Lie algebra of primitive elements of $\mathcal{D}_H$. Note also that, by Definition \ref{lanterndef}(1) and \cite[Lemma 9.2.4]{Mont}.
\begin{equation}\label{heart} \mathcal{L}(H) = P(\mathcal{D}_H) = \{f \in \mathcal{D}_H : f(\mathfrak{m}^2 + k1) = 0 \}.
\end{equation}

(2) Since $\mathcal{D}_H$ is a graded Hopf algebra, its subspace $\mathcal{L}(H)$ of primitive elements is a graded Lie algebra: for $d \geq 1$,
$$ \mathcal{L}(H)(d) =  \{f \in (H(d))^* : f(\mathfrak{m}^2 \cap H(d)) = 0 \}. $$
Thus, for all $d \geq 1$,
\begin{equation}\label{huh} \mathrm{dim}_k \mathcal{L}(H)(d) = \mathrm{dim}_k(H(d)/\mathfrak{m}^2 \cap H(d)) = \mathrm{dim}_k((\mathfrak{m}/\mathfrak{m}^2)(d)). \end{equation}
By Lemma \ref{hilbkey}(3), this completes the proof of (2).

(3) This follows from (\ref{heart}) and the definition of the Lie algebra of an algebraic group, \cite[$\S$9.1]{Humphreys}.

(4) This is \cite[Lemma 5.5]{And}.

(5) Appealing again to \cite[$\S$9.1]{Humphreys} and noting (\ref{heart}),
\begin{equation}\begin{aligned}\label{grunt}\mathfrak{y}_T =& \{f \in \mathcal{L}(H) : f((\mathrm{gr}T)^+ \mathrm{gr}H) = 0 \}\\ =& \{f \in \mathcal{D}_H : f((\mathrm{gr}T)^+ \mathrm{gr}H + \mathfrak{m}^2 + k1) = 0 \}.\end{aligned}\end{equation}
The convolution product makes $\mathcal{D}_T$ into a left module over $\mathcal{D}_H = U(\mathcal{L}(H))$: namely, for $u \in U(\mathcal{L}(H))$, $f \in \mathcal{D}_T$ and $t \in \mathrm{gr}T$,
$$ uf(t) = u(t_1)f(t_2). $$
It is clear that, with respect to this action, the restriction map
$$ \rho :U(\mathcal{L}(H)) = (\mathrm{gr}H)^* \longrightarrow \mathcal{D}_T = (\mathrm{gr}T)^* $$
is a homomorphism of left $U(\mathcal{L}(H)$-modules. From (\ref{grunt}), $\mathfrak{y}_T \subseteq \mathrm{ker}\rho$, and hence
\begin{equation}\label{kernel} U(\mathcal{L}(H))\mathfrak{y}_T \subseteq \mathrm{ker}\rho. \end{equation}
It remains to check that equality holds in (\ref{kernel}). Take a graded basis $\{u_1, \ldots , u_n \}$ of $\mathcal{L}(H)$, where $\{u_1, \ldots , u_m\}$ is a dual basis to a set $x_1, \ldots x_m$ of graded polynomial generators of $\mathrm{gr}T$ and where $\{u_{m+1}, \ldots , u_n \}$ forms a basis of $\mathfrak{y}_T$. Let $x_i \in H(d_i)$, so $u_i \in H(d_i)^*$ for $i = 1, \ldots , n.$ Thus $U(\mathcal{L}(H)) = \mathcal{D}_H$ is graded, where, for $j \geq 0$, a basis of $U(\mathcal{L}(H))(j)$ is given by the ordered monomials $u_1^{r_1} \ldots u_n^{r_n}$ for which $\sum_{i=1}^n r_i d_i = j$. Then $U(\mathcal{L}(H)\mathfrak{y}_T(j)$ is spanned by those ordered monomials in the $u_i$ of degree $j$ for which $r_i > 0$ for some $i > m$. Comparing the dimensions of $\mathcal{D}_T(j)$ with $(U(\mathcal{L}(H)/U(\mathcal{L}(H)\mathfrak{y})(j)$ for $j \geq 0$ now yields equality in (\ref{kernel}).

(6) It follows from (\ref{huh}) and (\ref{grunt}) that
 \begin{equation}\begin{aligned} \label{hohum} (\mathcal{L}(H)/\mathfrak{y})(d) \cong& \mathcal{L}(H)(d)/\mathfrak{y}(d)\\ \cong& \{f \in (H(d))^* : f((\mathfrak{m}^2 + (\mathrm{gr}T)^+ \mathrm{gr}H) \cap H(d)) = 0 \}\\ \cong& ((\mathrm{gr}T)^+/(\mathfrak{m}^2 \cap (\mathrm{gr}T)^+)(d))^*. \end{aligned} \end{equation}
Since the final term above is $\mathcal{L}(T)(d)$, this proves (6).

(7) The equivalence of (a) and (c) was noted in Lemma \ref{hilbkey}(4) and its proof. The equivalence of (a) and (b) is (\ref{huh}).
\end{proof}

In the literature, a finitely generated positively graded (and hence nilpotent) Lie algebra which is generated in degree 1 is called a \emph{Carnot Lie algebra}; they are important in a number of branches of mathematics, for example in Riemannian geometry. For a brief review with references, one can consult for instance \cite{Corn}

The \emph{classical} part of the picture described by the theorem, that is the \emph{connected cocommutative} aspect familiar from basic Lie theory, is as follows.

\begin{corollary} Retain the notation of the above theorem. Then the following are equivalent:
\begin{enumerate}
\item $\mathcal{L}(H)$ is abelian;
\item $\mathcal{L}(H) = \mathcal{L}(H)(1)$;
\item $H$ is cocommutative;
\item $H \cong U(\mathfrak{g})$ as a Hopf algebra, for some $n$-dimensional Lie algebra $\mathfrak{g}$;
\item $\mathrm{gr}H$ is cocommutative;
\item $W$ is abelian, $W \cong (k,+)^n$;
\item $\sigma (H) = (1^{(n)}).$
\end{enumerate}
\end{corollary}
\begin{proof} (1)$\Rightarrow$ (2): If $\mathcal{L}(H)$ is abelian, then $\mathcal{L}(H) = \mathcal{L}(H)(1)$ by Proposition \ref{litup}(4).

(3)$\Leftrightarrow$(4) This is a special case of the Cartier-Kostant-Gabri\`{e}l theorem \cite[Theorem 5.6.5]{Mont}.

(5)$\Leftrightarrow$(6): That $W$ is abelian if and only if its coordinate ring $\mathrm{gr}H$ is cocommutative is immediate by duality. That, in this case, $W \cong (k, +)^n$ is a consequence of the structure of abelian algebraic groups in charactersitic 0, \cite[Corollary 17.5, Exercises 15.11, 17.7]{Humphreys}.

(2)$\Leftrightarrow$(7): This is a special case of Proposition \ref{litup}.

(7)$\Rightarrow$(4): Suppose $\sigma (H) = (1^{(n)})$. Then $\mathrm{gr}H$ is generated by elements of $H(1)$, that is, by primitive elements. Hence $\mathrm{gr}H$ is cocommutative.

(3)$\Rightarrow$(5): Trivial.

(5)$\Rightarrow$(7): If $\mathrm{gr}H$ is cocommutative, then by the Cartier-Kostant-Gabri\`{e}l theorem \cite[Theorem 5.6.5]{Mont}it is an enveloping algebra as a Hopf algebra, so generated by the space $H(1)$ of primitive elements. That is, $\sigma (H) = (1^{(n)})$.
\end{proof}

We return to the family of examples $B(\lambda)$ to illustrate aspects of Proposition \ref{litup}.

\begin{example}\label{lambdasig} \textbf{Signature and lantern of} $B(\lambda)$. The connected Hopf algebras $B(\lambda)$, for $\lambda \in k$, were recalled from \cite{Zhuang} in \S\ref{smallHopf}. They constitute one of two infinite families of connected Hopf $k$-algebras of Gel'fand-Kirillov dimension 3. Starting from the description of the coideal subalgebras of $B(\lambda)$ in Proposition \ref{Blambda}, one easily calculates the following facts:
\begin{enumerate}
\item $\sigma (B(\lambda)) = (1^{(2)},2)$.
\item $\mathcal{L}(B(\lambda))$ is the Heisenberg Lie algebra of dimension 3; equivalently, the group $W$ with coordinate ring $\mathrm{gr}B(\lambda)$ is the 3-dimensional Heisenberg group. Here, $\mathrm{gr}B(\lambda) = k[X,Y,Z]$ in the obvious ``lazy" notation for lifts of elements to the associated graded algebra.
\item The unique two-dimensional Hopf subalgebra of $B(\lambda)$, namely $k\langle X, Y : [X,Y] = Y \rangle$, labelled $L_0 = R_0$ in Proposition \ref{Blambda}, has signature $(1^{(2)})$, with $\mathrm{gr}L_0 = k[X,Y]$.
\item The remaining two-dimensional left and right coideal subalgebras $L_{\beta}, R_{\beta}$, for $\beta \in \{\infty\} \cup (k\setminus\{0\})$, have signature $(1,2)$. For all such $\beta$, $\mathrm{gr}L_{\beta} = k[Y,Z]$ and $\mathrm{gr}R_{\beta} = k[Y,Z- XY]$.
\end{enumerate}
\end{example}

\subsection{Numerology}\label{numerology}

The first of the two results in this subsection assembles what we know about the signature of a connected Hopf algebra $H$. The second theorem gives a parallel account of the known numerical constraints on the signature of a quantum homogeneous space of such an $H$. For convenience, some results obtained earlier in the paper are restated here.

\begin{theorem}\label{Hopfnumbers} Let $k$ be an algebraically closed field of characteristic 0, and let $H$ be a connected Hopf $k$-algebra of finite Gel'fand-Kirillov dimension $n$. Let
$$ \sigma (H) = (d_1^{(m_1)}, \ldots , d_u^{(m_u)}), \textit{  so  } h_{\mathrm{gr}H}(t) = \prod_{i=1}^u \frac{1}{(1-t^{d_i})^{m_i}}. $$
\begin{enumerate}
\item (Wang, Zhang, Zhuang, \cite[Lemma 1.3(d)]{WZZ4}) If $n > 1$, then $m_1 \geq 2$. That is, $\mathrm{dim}_k P(H) \geq 2$ if $\mathrm{GKdim}H \geq 2.$
\item (NO GAPS) $\{d_1, \ldots , d_t\} = \{1, \ldots , t \}.$
\item For all $i = 1, \ldots , t$,
$$m_i \leq \frac{1}{i}\sum_{d|i}\mu(d)m_{1}^{(i/d)},$$
where $\mu:\mathbb{N}\rightarrow \mathbb{N}$ is the Mobius function.
\end{enumerate}
\end{theorem}
\begin{proof} (1) Since $\mathcal{L}(H)$ is generated in degree 1 by Proposition \ref{litup}(4), $\mathcal{L}(H)$ and hence $H$ would be one-dimensional if $\dim_k\mathcal{L}(H) = 1$.

(2) A lemma on $\mathbb{N}$-graded Lie algebras which are generated in degree 1, proved easily by induction, implies that, for all $i \geq 1$,
$$ \mathcal{L}(H)(i+1) = [\mathcal{L}(H)(1),\mathcal{L}(H)(i)]. $$
Thus (2) follows from this and Proposition \ref{litup}(2).

(3) Let $1 \leq i \leq t$. By (2) and Proposition \ref{litup}(7), $m_i = \mathrm{dim}_k\mathcal{L}(H)(i)$, so the bound follows Proposition \ref{litup}(4) and from the well-known Witt formula for the dimension of the $i$th graded summand of the free Lie algebra on $m_1$ generators of degree 1, \cite{Witt}, \cite{Se}.
\end{proof}

\begin{theorem}\label{numbers} Let $k$ and $H$ be as in Theorem \ref{Hopfnumbers}. Let each of $K$ and $L$ be right or left coideal subalgebras of $H$, of Gel'fand-Kirillov dimensions $m$ and $\ell$ respectively, with $L \subseteq K$.
Let $\sigma(K) = (e_1^{(r_1)}, \ldots , e_s^{(r_s)})$ and $\sigma(L) = (f_1^{(q_1)}, \ldots ,f_p^{(q_p)})$.
\begin{enumerate}
\item $ m = \sum_i r_i \geq \sum_j q_j = \ell.$
\item $\ell = m$ if and only if $L = K$.
\item $\sigma (L)$ is a sub-multiset of $\sigma (K)$. That is,
$$ h_{\mathrm{gr}L} (t) |h_{\mathrm{gr}K} (t).$$
Equality holds (of multisets and of Hilbert polynomials) if and only if $L=K$.
\item If $K \neq k$ then $e_1 = 1$. Similarly, of course, for $L$.

\end{enumerate}
\end{theorem}
\begin{proof}(1) Immediate from Lemma \ref{Tissiglan} and the definition of the signature, Definition \ref{sigdef}(4).

(2) This is a consequence of the corresponding result when $H$ is commutative, Theorem \ref{flag}(4), together with Lemma \ref{Tissiglan}.

(3) Immediate from the defition and from Lemma \ref{hilbkey}(4).

(4) This is Lemma \ref{primitivecoideal}.

\end{proof}

\begin{remarks} (1) One might expect that Theorem \ref{Hopfnumbers}(1) applies more generally, namely to quantum homogeneous spaces rather than just to Hopf algebras, especially in the light of Theorem \ref{numbers}(4). But this is not the case, as is illustrated by $H = B(\lambda)$, see Example \ref{lambdasig}(4).

(2) Similarly, the No-Gaps Theorem \ref{Hopfnumbers}(3) does \emph{not} extend to quantum homogeneous spaces. This is shown by the example below.
\end{remarks}

\begin{example}\label{gaps} \textbf{ A quantum homogeneous space with signature $(1^{(2)},3)$.} The example is one of the families listed in the classification of connected Hopf algebras of Gel'fand-Kirillov dimension 4, \cite[\textrm{Example 4.5}]{WZZ4}. Let $a,b,\lambda_{1},\lambda_{2}\in k$, and let $E$ be the $k$-algebra generated by $X, Y, Z, W$, subject to the following relations:
$$\begin{aligned}
\; [Y,X]&=[Z,Y]= 0, \; [Z,X]=X, \; [W,X]=aX,\; \\
  [W,Y]&=bX, \; [W,Z]=aZ-W+\lambda_{1}X+\lambda_{2}Y.
\end{aligned}$$
It is shown in \cite{WZZ4} that there is a Hopf algebra structure on $E$ such that $E$ is connected with $\mathrm{GKdim}E = 4$. Namely, one defines $X,Y,Z, W \in \mathrm{ker}\epsilon$, and $\Delta:E\rightarrow E\otimes E$ is fixed by  setting $X, Y \in P(E)$ and
$$\begin{aligned}
\Delta(Z)=1\otimes Z+X\otimes Y&-Y\otimes X+Z\otimes 1, \\
\Delta(W)=1\otimes W+W\otimes 1+ & Z\otimes X-X\otimes Z+X\otimes XY+XY\otimes X. \\
\end{aligned} $$
In \cite[Proposition 4.8]{WZZ4} it is shown that
$$ \sigma (E) = (1^{(2)},2,3);$$
indeed one can see from the definition of $\Delta$ that $Z \in E_2 \setminus E_1$ and $W \in E_3 \setminus E_2$.
Define $T=k\langle X,Y, W-XZ \rangle$ and confirm easily that
$$T=k[X,Y][(W-XZ); \partial],$$
where $\partial(X)=aX-X^{2}$ and $\partial(Y)=bX$. Note that $Z\notin T$ and that
$$\begin{aligned}
\Delta(W-XZ)&=1\otimes (W-XZ)+(W-XZ)\otimes 1\\
&+2(XY\otimes X)+2(X\otimes Z)+X^{2}\otimes Y-Y\otimes X^{2} \\
&\in T\otimes E.
\end{aligned} $$
Thus $T$ is a right coideal subalgebra of $E$, with
$$ \sigma (T) = (1^{(2)},3).$$
\end{example}

\section{Questions and Discussion}\label{open}

Some questions concerning connected Hopf algebras are listed in the survey article \cite{Pal}. We don't repeat those questions here, focussing instead on the possible role of quantum homogeneous homogeneous spaces in the study of these Hopf algebras. As elsewhere in this paper, $k$ is algebraically closed of characteristic 0.

\subsection{Classification of algebras}\label{class1}

\begin{question}\label{q1} Let $T$ be a $k$-algebra with a finite dimensional filtration $\mathcal{F} = \{T_i : i \geq 0\}$ with $T_0 = k$, such that $\mathrm{gr}_{\mathcal{F}}T$ is a commutative polynomial algebra in finitely many variables. What conditions on $T$  and $\mathcal{F}$ are required to ensure that $T$ is a quantum homogeneous space of a connected Hopf algebra with $\mathrm{GKdim}H < \infty$, with $\mathcal{F}$ the coradical filtration of $T$?
\end{question}

This question appears to be difficult. There are some easy observations to be made: some numerical constraints on $T$ can be read off from Theorem \ref{numbers}; and $T$ has to admit a 1-dimensional representation, the restriction of the counit. If the additional hypothesis is imposed, that $\mathrm{gr}_{\mathcal{F}}T$ is generated by finitely many elements of degree 1, then $T_1$ is a Lie algebra with respect to the bracket $[a,b] = ab - ba$. Hence, by the universal property of enveloping algebras, $T$ is an epimorphic image of $U(T_1)$ (assuming only $\mathrm{gr}T$ affine commutative, not necessarily a polynomial algebra). This observation is due to Duflo, \cite{Du}, \cite[Proposition 8.4.3]{MR}. Such algebras are called \emph{almost commutative} in \cite{MR}. But the Jordan plane is a quantum homogeneous space of a connected Hopf algebra, by Proposition \ref{Blambda}. Since the Jordan plane is shown in \cite[Proposition 14.3.9]{MR} \emph{not} to be almost commutative, it follows that the class of $k$-algebras defined by Question \ref{q1} is not a subclass of the class of almost commutative algebras.

One can refine Question \ref{q1} in various ways. For example, one can ask what is needed so that $T$ is a connected Hopf algebra, not just a quantum homogeneous space. In particular, we asked in \cite[Question L]{Pal} whether every such $T$ is the enveloping algebra of a finite dimensional Lie algebra. The answer to this is "No", as we shall demonstrate in \cite{BGZ}.

In a second direction, one can restrict the size of $T$ in Question \ref{q1} and ask for a classification. As noted in \ref{tiny}, if $\mathrm{GKdim}T \leq 1$, then $T$ is either $k$ or $k[x]$, with $x$ primitive. Beyond dimension 1, the question is open:

\begin{question}\label{q2} What are the quantum homogeneous spaces $T$ with $\mathrm{GKdim}T = 2$ in connected Hopf $k$-algebras $H$ of finite Gel'fand-Kirillov dimension?
\end{question}

If this is too easy, one can consider the same question for $H$ a Hopf domain of finite Gel'fand-Kirillov dimension, not necessarily connected. Note that the corresponding question for $T= H$  of dimension 2 has been answered in \cite{GZ}.

\subsection{Classification via the signature}\label{qsig}
The signature may offer a useful way to organise the possible $H$ and the possible $T$. So, for example, in the sense that all finite dimensional Lie $k$-algebras are ``known", we can describe all the connected Hopf $k$-algebras with $H$ with $\sigma(H) = (1^{(n)})$ - namely, they are the enveloping algebras $H = U(\mathfrak{g})$ with $\mathfrak{g} = P(H)$. And the same holds for the quantum homogeneous spaces $T$ with this signature, by $\S$\ref{cocommutative}. Next up might be the \emph{primitively thick} connected Hopf algebras $H$, these being the ones with signature $\sigma (H) = (1^{(n-1)},2)$. These are completely described, along with their coideal subalgebras, in \cite{BGZ}. There are various options as to what to consider next; for example, one can aim to classify the \emph{primitively thin} algebras, those at the opposite extreme from the thick ones. Namely:

\begin{question}\label{qthin}Classify the connected Hopf $k$-algebras $H$ with $\sigma (H) = (1^{(2)}, \ldots )$. Do the same for the quantum homogeneous spaces $T$ with $\sigma (T) = (1, \ldots )$.
\end{question}

Note that the second part of Question \ref{qthin} incorporates Question \ref{q2}.

\subsection{Quantisations of unipotent groups} \label{qunipotent}

Recall that if $H$ is connected Hopf $k$-algebra of finite Gel'fand-Kirillov dimension, then $\mathrm{gr}H$ is the coordinate ring of a unipotent $k$-group $U$. Naturally, one should consider reversing this passage to the associated graded algebra of $H$:

\begin{question}\label{q3} Which unipotent $k$-groups $U$ admit non-trivial ``lifts"? That is, for which such $U$ does there exist a \emph{noncommutative} connected Hopf algebra $H$ with $\mathrm{gr}H \cong \mathcal{O}(U)$?
\end{question}

Of course the answer to Question \ref{q3} is known when $U$ is abelian: if $U \cong (k,+)^n$ for any $n>1$, one can take $H = U(\mathfrak{g})$ where $\mathfrak{g}$ is any non-abelian Lie algebra of dimension $n$. So perhaps a first step on the route towards answering Question \ref{q3} might be:

\begin{question}\label{q4} Is there a unipotent group $U$ with $\mathrm{dim}U > 1$ and with $\mathrm{Lie}U$ Carnot, which has no non-trivial lift?
\end{question}

\subsection{Complete flags of quantum homogeneous spaces}\label{qflag}

If $H$ is a \emph{commutative} connected Hopf $k$-algebra of finite Gel'fand-Kirillov dimension $n$, then, as we noted in Theorem \ref{flag}, there is chain of $n+1$ coideal subalgebras (in fact Hopf subalgebras) from $k$ to $H$. The same is not always true in the \emph{cocommutative} case, where (by an easy argument making use of \cite[$\S$3.1, Examples (iii),(iv)]{OHagan}), such a flag exists if and only if $\mathfrak{g}$ has solvable radical $\mathfrak{r}$ with $\mathfrak{g}/\mathfrak{r}$ isomorphic to a finite direct sum of copies of $\mathfrak{sl}(2,k)$. Bearing these cases in mind and aiming to develop a generalisation of the solvable radical in the connected Hopf setting, one might propose the following:

\begin{question}\label{q5} (i) Is there a good structure theory for connected Hopf $k$-algebras $H$ which possess a complete flag of coideal subalgebras $K_i$,
$$ k = K_0 \subset K_1 \subset \cdots \subset K_n = H, $$
with $K_i^+H$ a Hopf ideal of $H$?

(ii) If $H$ is \emph{any} connected Hopf $k$-algebra of finite Gel'fand-Kirillov dimension, does $H$ have a maximal $\mathrm{ad}H$-normal Hopf subalgebra $R$ with the property (i)?

(iii) Given (ii), what can be said about the Hopf algebra $H/R^+H$?

(iv) How does the class of Hopf algebras in (i) compare with the IHOEs studied in \cite{OHagan}?
\end{question}

\end{document}